\documentclass[12pt]{amsart}
\usepackage{amssymb,amsmath,amsthm,amsfonts,amsopn,url,color,hyperref,enumerate,mathtools}
\usepackage[normalem]{ulem}
\usepackage{todonotes}
\usepackage{stmaryrd}
\usepackage{enumerate,hyperref}
\usepackage{amsmath,amsfonts,amssymb,amsthm,graphicx,url}
\usepackage{amscd}
\usepackage[mathscr]{euscript}
\usepackage{mathtools,rotating}
\usepackage[all,2cell]{xy}
\usepackage{esvect}
\usepackage{tabularx}
\usepackage{arydshln}
\UseAllTwocells
\input xy
\xyoption{all}
\usepackage{pgf,tikz}
\usepackage{mathrsfs}
\usetikzlibrary{arrows}
\usetikzlibrary{patterns}
\usepackage{tikz,pgfplots}
\usetikzlibrary{matrix, cd}
\usepackage{esvect}
\usepackage{tabularx}
\usepackage{arydshln}
\usepackage{array}
\newtheorem{theorem}{Theorem}[section]

\newtheorem{lemma}[theorem]{Lemma}
\newtheorem{proposition}[theorem]{Proposition}
\newtheorem{observation}[theorem]{Observation}
\newtheorem{corollary}[theorem]{Corollary}
\theoremstyle{definition}\newtheorem{remark}[theorem]{Remark}
\newtheorem*{question*}{Question}
\theoremstyle{definition}\newtheorem{definition}[theorem]{Definition}
\newtheorem{example}[theorem]{Example}

\author{Hannah Klawa}
\address{Department of Mathematics \\ Indiana University East \\ Richmond, IN 47374}
\email{hklawa@iue.edu}

\title{Graded Perinormality}

\begin{document}

\begin{abstract} An integral domain $R$ is \emph{perinormal} if every local going-down overring is a localization of $R$ and \emph{globally perinormal} if every going-down overring is a localization of $R$. In this paper, I introduce notions of graded perinormal and graded globally perinormal domains and show that many results obtained for perinormal and globally perinormal domains have graded analogs. I also give some results for descent of properties between a graded domain and its $0$th graded component.
\end{abstract}

\maketitle
\section{Introduction}

In~\cite{EpSh-peri}, Epstein and Shapiro introduced the notions of perinormality and global perinormality. An integral domain $R$ is \emph{perinormal} if every local going-down overring is a localization of $R$ and \emph{globally perinormal} if every going-down overring is a localization of $R$. Epstein and Shapiro showed that perinormality is close to normality in the sense that Noetherian normal implies perinormal which implies weakly normal~\cite[Corollary 3.4, Theorem 3.10]{EpSh-peri}. The purpose of this paper is to introduce graded notions of perinormality and global perinormality and show that many results that hold for perinormality and global perinormality have graded analogs. In the non-graded setting, McCrady and Weston in~\cite{McWe-periext} and Dumitrescu and Rani in~\cite{DuRa-peripoly} have studied the question posed by Epstein and Shapiro in~\cite{EpSh-peri} of whether for an integral domain $A$, $A[X]$ being a perinormal domain implies $A$ is perinormal. I give some results for descent of properties between a graded domain and its $0$th graded component and utilize them to consider the graded analog of the polynomial ring question that has been considered in~\cite{McWe-periext} and~\cite{DuRa-peripoly}.

\section{Graded overrings}

Recall that a \emph{torsionless grading monoid} is a torsion-free commutative cancellative monoid. Let $\Gamma$ be a torsionless grading monoid. An integral domain $R$ is a \emph{$\Gamma$-graded domain} if
$$R = \bigoplus_{\alpha \in \Gamma}R_{\alpha}$$
where each $R_{\alpha}$ is an additive subgroup of $R$ and $R_{\beta}R_{\gamma} \subseteq R_{\beta + \gamma}$ for all $\beta, \gamma \in \Gamma$. A nonzero element $r \in R_{\alpha}$ is \emph{homogeneous} of degree $\alpha$ and each $R_{\alpha}$ is a \emph{graded component} of $R$.

Given a $\Gamma$-graded integral domain $$R = \bigoplus_{\alpha \in \Gamma}R_{\alpha},$$ $R_H$ where 
$$H = \{x \in R\, | \, x \neq 0\text{ is homogeneous}\}$$ is the \emph{graded quotient field} of $R$ and is a $\langle \Gamma \rangle = \{a - b\, | \, a, b \in \Gamma\}$ graded domain where 
$$(R_H)_{\alpha} = \left\lbrace\tfrac{f}{g}\,|\,f,g \textnormal{ are homogeneous}, g \neq 0\textnormal{ and }\deg(f) - \deg(g) = \alpha\right\rbrace$$
for $\alpha \in \langle \Gamma \rangle$.

The graded (or homogeneous) primes of $R$ will be denoted $h\textnormal{-Spec}(R)$ and the graded maximal ideals of $R$ will be denoted $h\textnormal{-Max}(R)$. A graded integral domain $R$ with a single graded maximal ideal is gr-local.

Unless otherwise stated $G$ will always denote a torsion free abelian group and we will assume all graded domains are graded by a torsion free abelian group unless otherwise stated.

It will be useful later on to be able to relate an overring of the $0$th graded component of a graded integral domain to an overring of the graded integral domain. Given an overring $A$ of the $0$th graded component $R_0$ of a graded integral domain $R = \bigoplus\limits_{\alpha \in G} R_{\alpha}$, let
$$AR_{\alpha} \coloneqq \left\lbrace \sum_{\textnormal{finite}}a_ir_i \,\Big\vert \, a_i \in A,\,r_i\in R_{\alpha}\right\rbrace.$$
It is elementary to show that $AR\coloneqq \bigoplus\limits_{\alpha \in G} AR_{\alpha} \subseteq R_H$ is a graded overring of $R$. If $A$ or $R$ is flat over $R_0$, it is elementary to show that $AR \cong A \otimes_{R_0} R$ as $R_0$-algebras and flat over $R$.

If neither $A$ nor $R$ is not flat over $R_0$, $AR$ need not be isomorphic to $A \otimes_{R_0} R$ as the example below demonstrates.

\begin{example} Let $R = k[x^2, x^3] + yk[x,y]$ graded by the degree of $y$. So $R_0 = k[x^2, x^3]$. Consider the overring $A = k[x]$ of $R_0$. Then $A \otimes_{R_0} R$ is not an integral domain. Consider the elements $\alpha = x^2 \otimes y - x \otimes xy$ and $\beta = x^5\otimes 1$. Then 
\begin{eqnarray*}
\alpha\beta & = & (x^2 \otimes y - x \otimes xy)(x^5 \otimes 1)\\
& = & (x^7 \otimes y) - (x^6 \otimes xy)\\
& = & (x^4 \otimes x^3y) - (x^4 \otimes x^3y)\\
& = & 0
\end{eqnarray*}
but neither $\alpha$ nor $\beta$ are zero in $A \otimes_{R_0} R$ as is shown below.

\bigskip
Let $$\psi: A \times R \longrightarrow R_0\slash (x^2,x^3)\cong k$$ be given by 
$$(a_0 + a_1x + \cdots + a_nx^n, b_{00} + b_{01}y  + b_{11}xy + \cdots + b_{jk}x^jy^k) \longmapsto a_1b_{11}.$$
This is clearly well-defined. Let $f_1 = a_0 + a_1x + \cdots + a_nx^n, f_2 = d_0 + d_1x + \cdots + d_mx^m \in A$, $g_1 = b_{00} + b_{01}y  + b_{11}xy + \cdots + b_{jk}x^jy^k, g_2 = c_{00} + c_{01}y + c_{11}xy + \cdots + c_{i\ell}x^iy^{\ell} \in R$ and $r = r_0 + r_2x^2 + \cdots + r_sx^s, t = t_0 + t_2x^2 + \cdots + t_wx^w \in R_0$ be arbitrary. Then
\begin{eqnarray*}
\psi(rf_1 + tf_2, g_1) & = &(r_0a_1 + t_0d_1)b_{11}\\
& = & r_0a_1b_{11} + t_0d_1b_{11}\\
&=& r\psi(f_1,g_1) + t\psi(f_2, g_1)
\end{eqnarray*}
and
\begin{eqnarray*}
\psi(f_1, rg_1 + tg_2) & = & a_1(r_0b_{11} + t_0c_{11})\\
& = & a_1r_0b_{11} + a_1t_0c_{11}\\
& = & r_0a_1b_{11} + t_0a_1c_{11}\\
& = & r\psi(f_1, g_1) + t\psi(f_1, g_2).
\end{eqnarray*}
So $\psi$ is an $R_0$-bilinear map. So, by~\cite[Corollary 12 p. \!\!\!368]{dummit}, there is an $R_0$-module homomorphism
$$\Psi: A \otimes_{R_0} R \longrightarrow R_0\slash (x^2,x^3)$$
given by 
$$(a_0 + a_1x + \cdots + a_nx^n) \otimes (b_{00} + b_{01}y  + b_{11}xy + \cdots + b_{jk}x^jy^k) \longmapsto a_1b_{11}.$$
Note that
$$\Psi(x^2 \otimes y - x \otimes xy) = -1$$
and hence $x^2 \otimes y - x \otimes xy \neq 0 \in A \otimes_{R_0} R$. 
\bigskip

It is clear that the map $$\varphi: A \times R \longrightarrow k[x,y]$$
given by 
$$(f,g) \longmapsto fg$$
is a well-defined $R_0$-bilinear map. So, by~\cite[Corollary 12 p. \!\!\!368]{dummit}, there is an $R_0$-module homomorphism 
$$\Phi: A \otimes_{R_0}R \longrightarrow k[x,y]$$
given by 
$$f \otimes g\longmapsto fg.$$
Note that $\Phi(x^5 \otimes 1) = x^5$ and hence $x^5 \otimes 1 \neq 0 \in A \otimes_{R_0} R$.

\end{example}

\subsection{Graded going-down}

In~\cite[Definition 2.1]{Sah-gradedgd}, Sahandi and Shirmohammadi introduce a notion of graded going-down for an extension of integral domains. The same notion, restricted to graded overrings, is introduced in~\cite[Definition 2.1.2]{Kla-dissertation} as ${}^\ast$going-down. We utilize the terminology of Sahandi and Shirmohammadi in what follows.

\begin{definition} Let $R$ be a $G$-graded domain. A graded overring $S$ of $R$ \emph{satisfies graded going-down $($or gGD$)$} or is a \emph{graded going-down overring $($or gGD overring$)$} if whenever $P, Q \in  h\textnormal{-Spec}(R)$ and $\tilde{Q} \in  h\textnormal{-Spec}(S)$ such that $P \subseteq Q$ and $\tilde{Q} \cap R = Q$, there exists $\tilde{P} \in  h\textnormal{-Spec}(S)$ such that $\tilde{P} \subseteq \tilde{Q}$ and $\tilde{P} \cap R = P$.
\end{definition}

It is clear that if $S$ is a graded overring of $R$ that satisfies going-down, then $S$ satisfies gGD.

\begin{observation}\label{gr-goingdown-surj} Let $R$ be a graded domain and $S$ a graded overring of $R$. Then $S$ satisfies gGD over $R$ if and only if the map $$h\textnormal{-Spec}(S_{(Q)}) \longrightarrow h\textnormal{-Spec}(R_{(Q \cap R)})$$ is surjective for each $Q \in h\textnormal{-Spec}(S)$. 
\end{observation}

\begin{proof} The proof is parallel to the well-known non-graded result with appropriate graded substitutions.
\end{proof}

\begin{observation} Let $R$ be a $G$-graded domain and $T$ a graded overring of $R$. If $R \subseteq T$ satisfies gGD and $W$ is a homogeneous multiplicative subset of $T$, then $R_{W \cap R} \subseteq T_{W}$ satisfies gGD.
\end{observation}
\begin{proof} The proof is parallel to the non-graded version making appropriate graded substitutions.
\end{proof}

\subsection{Flat graded overrings}

It may be well-known that the graded version of Richman's flatness criteria~\cite[Theorem 2]{Ric-genqr} given below holds, but we will prove it in this section for completeness since it will be used several times. We will use the following results.

\begin{proposition}[\cite{Her-2notesflat}, Proposition 3.1]\label{Herrmannflat} Let $R$ be a $G$-graded ring where $G$ is an \textnormal{(}additive\textnormal{)} abelian group. Then for any graded $R$-module $M$, $M$ is flat over $R$ if and only if for any exact sequence of graded $R$-modules and graded homomorphisms 
$$0 \longrightarrow N^{\prime} \longrightarrow N \longrightarrow N^{\prime\prime} \longrightarrow 0,$$
the induced sequence 
$$0 \longrightarrow N^{\prime} \otimes_R M \longrightarrow N \otimes_R M \longrightarrow N^{\prime\prime} \otimes_R M \longrightarrow 0$$
is exact.
\end{proposition}

\begin{lemma}\label{lemma-gr-rich} Let $R$ be a $G$-graded domain where $G$ is an torsion-free abelian group and $S$ a graded $R$-algebra. Then $S$ is flat if and only if for all $\mathfrak{n} \in  h\textnormal{-Max}(S)$, $S_{(\mathfrak{n})}$ is flat over $R_{(\mathfrak{n} \cap R)}$.

\begin{proof} Let $R$ be a $G$-graded domain where $G$ is an abelian group and $S$ a graded $R$-algebra. If $S$ is flat, then $S_{(\mathfrak{n})}$ is flat over $R_{(\mathfrak{n} \cap R)}$ for all $\mathfrak{n} \in  h\textnormal{-Max}(S)$.
\bigskip
Suppose that $S_{(\mathfrak{n})}$ is flat over $R_{(\mathfrak{n} \cap R)}$ for all $\mathfrak{n} \in  h\textnormal{-Max}(S)$. Suppose there exists a graded inclusion of graded $R$-modules $N \hookrightarrow M$ such that 
$$\varphi: S \otimes_R N \longrightarrow S \otimes_R M$$ is not injective. Let $K = \text{Ker}\,\varphi$. Let $\mathfrak{n} \in  h\textnormal{-Max}(S)$. Because $S_{(\mathfrak{n})}$ is flat over $R_{(\mathfrak{n} \cap R)}$
$$(S \otimes_R N)_{(\mathfrak{n})} \cong S_{(\mathfrak{n})}\otimes_{R_{(\mathfrak{n} \cap R)}} N_{(\mathfrak{n})} \overset{\varphi_{(\mathfrak{n})}}\longrightarrow S_{(\mathfrak{n})} \otimes_{R_{(\mathfrak{n} \cap R)}} M_{(\mathfrak{n})} \cong (S \otimes_R M)_{(\mathfrak{n})}$$
is injective by Proposition~\ref{Herrmannflat}. So $K_{(\mathfrak{n})} = 0$ for all $\mathfrak{n} \in  h\textnormal{-Max}(S)$. Thus $K = 0$. So $\varphi$ is injective. Hence $S$ is flat over $R$ by Proposition~\ref{Herrmannflat}.
\end{proof}
\end{lemma}

Using the above lemma, we obtain a graded version of~\cite[Theorem 2]{Ric-genqr}.

\begin{proposition}\label{gradedRichThm2}  Let $R$ be a $G$-graded domain where $G$ is a torsion-free abelian group. Let $S$ be a graded overring of $R$. Then the following are equivalent, 
\begin{enumerate} 
\item $S$ is a flat overring of $R$.
\item $S_{(\mathfrak{p})} = R_{(\mathfrak{p} \cap R)}$ for all $\mathfrak{p} \in h\textnormal{-Spec}(S)$.
\item $S_{(\mathfrak{n})} = R_{(\mathfrak{n} \cap R)}$ for all $\mathfrak{n} \in  h\textnormal{-Max}(S)$.
\end{enumerate} 

\begin{proof} ($1\implies 2$) Suppose that $S$ is flat over $R$. Let $\mathfrak{p} \in h\textnormal{-Spec}(S)$. Then $\mathfrak{p} \in \textnormal{Spec}(S)$. By~\cite[Theorem 2]{Ric-genqr}, $S_{\mathfrak{p}} = R_{\mathfrak{p} \cap R}$. Because $S_{(\mathfrak{p})} = R_H \cap S_{\mathfrak{p}}$ and $R_{(\mathfrak{p} \cap R)} = R_H \cap R_{\mathfrak{p} \cap R}$, it follows that $S_{(\mathfrak{p})} = R_{(\mathfrak{p} \cap R)}$.

($2 \implies 3$) A graded maximal ideal is a graded prime ideal so there is nothing to prove.

($3 \implies 1$) Suppose that $S_{(\mathfrak{n})} = R_{(\mathfrak{n} \cap R)}$ for all $\mathfrak{n} \in  h\textnormal{-Max}(S)$. Then $S_{(\mathfrak{n})}$ is flat over $R_{(\mathfrak{n} \cap R)}$ for all $\mathfrak{n} \in  h\textnormal{-Max}(S)$. Thus $S$ is flat over $R$ by Lemma~\ref{lemma-gr-rich}.

\end{proof}
\end{proposition}

This graded version of~\cite[Theorem 2]{Ric-genqr} is useful when working with graded overrings.

In~\cite{HaSa-grprINCext}, Sahandi and Hamdi define a graded overring $T$ of a graded domain $R$ (graded by a 
torsionless grading monoid) to be $h$-flat if for every graded prime ideal $\mathfrak{q}$ of $T$, $R_{\mathfrak{q} \cap T} = T_\mathfrak{q}$.

\begin{proposition} Let $R$ be a $\Gamma$-graded integral domain where $\Gamma$ is an arbitrary torsionless grading monoid. A graded overring $T$ of $R$ is $h$-flat over $R$ if and only if $T$ is flat over $R$.

\begin{proof} We may assume $R$ is $G$-graded where $G$ is a torsion-free abelian group by letting $R_{\gamma} = \{0\}$ for any $\gamma \in \{a - b \, |\, a, b \in \Gamma\}\backslash \Gamma$. Let $T$ be graded overring of $R$. Suppose $T$ is $h$-flat. Then for every $\mathfrak{q}\in h\textnormal{-Spec}(T)$, $R_{\mathfrak{q} \cap T} = T_\mathfrak{q}$. It follows that for every $\mathfrak{q}\in h\textnormal{-Spec}(T)$, $R_{(\mathfrak{q} \cap T)} = T_{(\mathfrak{q})}$. Hence $T$ is flat over $R$ by Proposition~\ref{gradedRichThm2}. 

That flatness implies $h$-flatness is clear by~\cite[Theorem 2]{Ric-genqr}
\end{proof}
\end{proposition}

\begin{remark} In~\cite[Corollary 3.6]{ACZ-grpruf}, Anderson et al show that a graded integral domain (graded by a 
torsionless commutative cancellative monoid) is a graded Pr\"{u}fer domain if and only if every graded overring is $t$-flat. Let $R$ be a graded Pr\"{u}fer domain, $S$ a graded overring of $R$ and $M \in  h\textnormal{-Max}(S)$. By~\cite[Theorem 3.5]{ACZ-grpruf}, it follows that $S_{(M)} = R_{(M \cap R)}$. Thus $S$ is flat over $R$ by Proposition~\ref{gradedRichThm2}. So, in a graded Pr\"{u}fer domain, every graded overring is flat.
\end{remark}

\subsection{Graded perinormality and graded global perinormality}

Recall that an integral domain $R$ is \emph{perinormal} if every local going-down overring is a localization of $R$ and \emph{globally perinormal} if every going-down overring is a localization of $R$. We define graded analogs of these notions and show that many of the results proved by Epstein and Shapiro in~\cite{EpSh-peri} for perinormal and globally perinormal domains have graded analogs.

\begin{definition} A $G$-graded domain $R$ is a \textbf{\emph{graded perinormal domain}} if every gr-local gGD graded overring is a localization of $R$.
\end{definition}

\begin{definition} A $G$-graded domain $R$ is a \emph{graded globally perinormal domain} if every gGD graded overring is a localization of $R$.
\end{definition}

With graded perinormality and graded global perinormality defined as above, we immediately obtain analogs of results from~\cite{EpSh-peri} by substituting in the graded analogs where appropriate in the proofs.

\begin{proposition}\label{gr-prop-2.5-1} Let $R$ be a $G$-graded domain. If $R_{(\mathfrak{m})}$ is a graded perinormal domain for all $\mathfrak{m} \in  h\textnormal{-Max}(R)$, then $R$ is a graded perinormal domain. If $R$ is a graded perinormal domain, then $R_W$ is a graded perinormal domain for every homogeneous multiplicative subset $W \subseteq R$.
\end{proposition}
\begin{proof} The proof is parallel to that of~\cite[Proposition 2.5]{EpSh-peri} making appropriate graded substitutions. 
\end{proof}

\begin{proposition}\label{gr-prop-6.1} Let $R$ be a $G$-graded domain. If $R$ is a graded globally perinormal domain, then $R_W$ is a graded globally perinormal domain for each homogeneous multiplicative subset $W$.
\end{proposition}

\begin{proof} Let $W$ a homogeneous multiplicative subset of $R$. Noting that $R_W$ is a graded overring of $R$ and any graded overring of $R_W$ is a graded overring of $R$, the proof is parallel to that of~\cite[Proposition 6.1]{EpSh-peri} using Observation~\ref{gr-goingdown-surj} and ~\cite[Proposition 1.3]{GiOh-qr} along with the usual appropriate graded substitutions. 
\end{proof}

\begin{proposition}\label{gradedperiprop2.4} Let $R$ be a $G$-graded domain. Then $R$ is a graded perinormal domain if and only if every graded overring that satisfies gGD is flat. 
\end{proposition}
\begin{proof} The proof is parallel to that of~\cite[Proposition 2.4]{EpSh-peri} making appropriate graded substitutions and using Proposition~\ref{gradedRichThm2}, the graded version of~\cite[Theorem 2]{Ric-genqr}, in place of~\cite[Theorem 2]{Ric-genqr}.
\end{proof}

\begin{proposition}\label{gradedperi-prop3.3} Let $(R,\mathfrak{m})$ be a gr-local graded perinormal domain. Let $S$ be a graded overring that satisfies gGD such that there exists $P \in  h\textnormal{-Spec}(S)$ such that $P \cap R = \mathfrak{m}$. Then $S = R$.
\end{proposition}
\begin{proof} Let $(R,\mathfrak{m})$ be a gr-local graded perinormal domain. Let $S$ be a graded overring that satisfies gGD such that there exists $P \in  h\textnormal{-Spec}(S)$ such that $P \cap R = \mathfrak{m}$. Note that $S_{(P)}$ is a gr-local graded overring that satisfies gGD over $R$ (because $S$ satisfies gGD over $R$ and $S_{(P)}$ satisfies going-down and hence gGD over $S$). Because $R$ is a graded perinormal domain, it follows that $S_{(P)}$ is a homogeneous localization of $R$. Hence $S_{(P)} = R_{(\mathfrak{m})}$. Thus
$$R = R_{(\mathfrak{m})} \subseteq S \subseteq S_{(P)} = R_{(\mathfrak{m})} = R.$$
So $S = R$.
\end{proof}

An integral domain $R$ is \emph{weakly normal} if for every integral overring $S$ such that the map $\textnormal{Spec}(S) \longrightarrow \textnormal{Spec}(R)$ is a bijection and for all $P \in \textnormal{Spec}(S)$, the field extension  $$R_{P \cap R} \slash (P \cap R)R_{P \cap R} \longrightarrow S_P \slash PS_P$$ is purely inseparable, $S = R$~\cite[Remark 1]{Yan-weaklynormal}. From Proposition~\ref{gradedperi-prop3.3}, we can prove that for a graded integral domain, graded perinormality implies weak normality.

\begin{proposition}\label{gradedperi-prop3.4} Let $R$ be a $G$-graded domain. If $R$ is a graded perinormal domain, then $R$ is weakly normal.
\end{proposition}

\begin{proof} The proof is parallel to that of~\cite[Corollary 3.4]{EpSh-peri} with the appropriate graded replacements where we assume that $R$ is gr-local because a local property is a graded local property and utilize Proposition~\ref{gradedperi-prop3.3} in place of~\cite[Proposition 3.3]{EpSh-peri}.
\end{proof}

\begin{proposition}\label{gradedperiprop6.2} Let $R$ be $G$-graded domain that is a graded perinormal domain. Then $R$ is a graded globally perinormal domain if and only if every graded flat overring of $R$ is a localization of $R$.
\end{proposition}

\begin{proof} The proof is parallel to that of~\cite[Proposition 6.2]{EpSh-peri} making appropriate graded substitutions and replacing the use of~\cite[Proposition 2.4]{EpSh-peri} with the graded version given in Proposition~\ref{gradedperiprop2.4} above. 
\end{proof}

\subsubsection{Graded $R_1$-domains}

Let $\Gamma$ be an arbitrary torsionless grading monoid. In~\cite{AAC-grval}, Anderson et al define a $\Gamma$-graded domain $R$ to be a \emph{graded valuation domain} (or simply a \emph{gr-valuation domain}) if $x$ or $x^{-1}$ is in $R$ for every nonzero homogeneous $x$ in the graded quotient field of $R$.

\begin{definition} A $G$-graded domain $R$ is a \emph{gr-$(R_1)$ domain} if $R_{(P)}$ is a gr-valuation domain for every $P \in h\textnormal{-Spec}^1(R)$. 
\end{definition}

\begin{proposition}\label{gr-valuation} Let $R$ be a $G$-graded domain. Then $R$ is a gr-$(R_1)$ domain if and only if $R_P$ is a valuation domain for every $P \in h\textnormal{-Spec}^1(R)$. 
\end{proposition}
\begin{proof} The result follows immediately from the definition of gr-$(R_1)$ and~\cite[Lemma 4.3]{Sah-chargrpvmd}.
\end{proof}

Clearly a $G$-graded domain that is an $(R_1)$ domain is a gr-$(R_1)$ domain.

\begin{lemma}\label{lem-gr-v-max} Let $R$ be a $G$-graded domain. For any $P \in h\textnormal{-Spec}(R)$, there is a graded overring $T$ that is a gr-valuation domain with graded maximal ideal that contracts to $P$.

\begin{proof} Let $P \in h\textnormal{-Spec}(R)$. By~\cite[Lemma 4.4]{ChOh-disvalgrNo}, there exists graded overring $V$ that is a gr-valuation domain with $Q \in  h\textnormal{-Spec}(V)$ such that $Q\cap R = P$. Then $T = V_{(P)}$ is a gr-valuation domain whose graded maximal ideal contracts to $P$.
\end{proof}
\end{lemma}

\begin{proposition}\label{gr-Prop3.2-peri} Let $R$ be a $G$-graded domain. If $R$ is a graded perinormal domain, then $R$ is a gr-$(R_1)$ domain.
\end{proposition}
\begin{proof} Using Lemma~\ref{lem-gr-v-max} and appropriate graded substitutions, the proof is parallel to that of~\cite[Proposition 3.2]{EpSh-peri}. 
\end{proof}

\begin{proposition}\label{gr-Prop3.9-peri} Let $R$ be a $\mathbb{Z}$-graded Noetherian gr-$(R_1)$ domain and let $S$ be a graded gGD overring of $R$. Then $S$ satisfies gr-$(R_1)$, and the map $\textnormal{Spec}(S) \longrightarrow \textnormal{Spec}(R)$ induces an injective map $ h\textnormal{-Spec}^1(S) \longrightarrow  h\textnormal{-Spec}^1(R)$ whose image consists of $\mathfrak{p} \in  h\textnormal{-Spec}^1(R)$ such that $\mathfrak{p}S \neq S$.
\end{proposition}
\begin{proof} The proof is parallel to that of~\cite[Proposition 3.9]{EpSh-peri} using appropriate graded replacements and using the Proposition~\ref{gr-valuation} characterization of gr-$(R_1)$. Adding the $\mathbb{Z}$-graded Noetherian assumption allows one to use the fact that graded height coincides with height (see~\cite[Theorem 1.5.8]{Bruns}) which is used when utilizing gGD. 
\end{proof}

In the following graded versions of~\cite[Lemma 4.1]{EpSh-peri} and~\cite[Lemma 4.2]{EpSh-peri}, ``generalized Krull" is replaced with ``Krull" because ``generalized Krull" is not needed for the purposes here and graded Krull domains have been studied in detail.

\begin{lemma}\label{gr-lem4.1-peri} Let $R$ be a $G$-graded gr-$(R_1)$ domain whose integral closure $R^{\prime}$ is a graded Krull domain and such that for all $P \in  h\textnormal{-Spec}^1(R^{\prime})$, $P \cap R \in  h\textnormal{-Spec}^1(R)$. If there exists $\mathfrak{m} \in  h\textnormal{-Max}(R)$ such that $\mathfrak{p} \subseteq \mathfrak{m}$ for all $\mathfrak{p} \in  h\textnormal{-Spec}^1(R)$, then $R$ is gr-local.
\end{lemma}
\begin{proof} The proof is parallel to that of~\cite[Lemma 4.1]{EpSh-peri} with appropriate graded replacements and using~\cite[Theorem 5.12]{AnAn-divgr} which gives that a graded Krull domain is equal to the intersection of its homogeneous localizations at height one graded prime ideals. 
\end{proof}

\begin{lemma}\label{gr-lem4.2-peri} Let $(R, \mathfrak{m})$ be a gr-local $\mathbb{Z}$-graded domain. Suppose that $R$ is a Noetherian gr-$(R_1)$ domain such that for all $P \in  h\textnormal{-Spec}^1(R^{\prime})$, $P \cap R \in  h\textnormal{-Spec}^1(R)$. Let $S$ be a graded integral overring of $R$ that satisfies gGD over $R$. Then $S$ is gr-local.

\begin{proof} The proof is parallel to that of~\cite[Lemma 4.2]{EpSh-peri} using appropriate graded substitutions and making use of the facts that the integral closure $R^{\prime}$ is a graded overring by~\cite[Lemma 1.6]{Cha-grintpruferlike} and graded height one primes of $R^{\prime}$ contract to graded height one primes of $R$.
\end{proof}
\end{lemma}

\subsubsection{Noetherian universally catenary domains}

\begin{theorem}\label{gr-thm4.7-peri} Let $R = \bigoplus_{\alpha\in G}R_{\alpha}$ be a $\mathbb{Z}$-graded Noetherian  universally catenary domain.
\begin{itemize}
\item[$(a)$] $R$ is a graded perinormal domain.
\item[$(b)$] $R$ is gr-$(R_1)$ and for each $\mathfrak{p} \in  h\textnormal{-Spec}(R)$, $R_{(\mathfrak{p})}$ is the only graded ring $S$ between $R_{(\mathfrak{p})}$ and $R_H$ such that the induced map $h\textnormal{-Spec}(S) \longrightarrow h\textnormal{-Spec}(R_{(\mathfrak{p})})$ is an order reflecting bijection.
\item[$(c)$] $R$ is gr-$(R_1)$ and for each $\mathfrak{p} \in  h\textnormal{-Spec}(R)$, $R_{(\mathfrak{p})}$ is the only graded ring $S$ between $R_{(\mathfrak{p})}$ and its integral closure such that the induced map $h\textnormal{-Spec}(S) \longrightarrow h\textnormal{-Spec}(R_{(\mathfrak{p})})$ is an order-reflecting bijection.
\end{itemize}

\begin{proof} Let $R = \bigoplus_{\alpha\in G}R_{\alpha}$ be a $G$-graded domain.

\bigskip
((a) $\implies$ (b)). Let $R = \bigoplus_{\alpha\in G}R_{\alpha}$ be a $G$-graded domain. Because graded perinormality is a gr-local property, suppose that $R$ is a gr-local graded perinormal domain. Then $R$ is gr-$(R_1)$ by Proposition~\ref{gr-Prop3.2-peri}. Let $S$ be a graded overring of $R$ such that $h\textnormal{-Spec}(S) \longrightarrow h\textnormal{-Spec}(R)$ is an order-reflecting bijection. Then $S$ satisfies gGD over $R$ and is a gr-local domain. Because $R$ is a graded perinormal domain and the graded maximal ideal of $S$ must contract to the graded maximal ideal of $R$, $S = R$ by Proposition~\ref{gradedperi-prop3.3}.

\bigskip

((b) $\implies$ (c)). This follows directly because the integral closure is contained within the graded quotient field by~\cite[Proposition 2.1]{AnAn-divgr}.

\bigskip

((c) $\implies$ (a)). Let $(S, \mathfrak{n})$ be a gr-local graded gGD overring of $R$. Let $\mathfrak{p} \coloneqq \mathfrak{n} \cap R$ (so $\mathfrak{p} \in  h\textnormal{-Spec}(R)$). Note that $R_{(\mathfrak{p})}$ satisfies gr-$(R_1)$ so by Proposition~\ref{gr-Prop3.9-peri}, there is a bijection $ h\textnormal{-Spec}^1(S) \longrightarrow  h\textnormal{-Spec}^1(R_{(\mathfrak{p})})$ where by~\cite[Lemma 3.7]{EpSh-peri} corresponding localizations are equal and hence corresponding graded localizations are equal. Note that $R_{(\mathfrak{p})}$ is universally catenary because $R$ is. By integrality,~\cite[Lemma 3.7]{EpSh-peri} and~\cite[Lemma 4.3]{EpSh-peri}, there is a bijection $ h\textnormal{-Spec}^1(R_{(\mathfrak{p})}^{\prime}) \longrightarrow  h\textnormal{-Spec}^1(R_{(\mathfrak{p})})$ where corresponding localizations are equal and hence corresponding graded localizations are equal. Note that $(R_{(\mathfrak{p})})^{\prime}$ is a Krull domain by the Mori-Nagata Integral Closure Theorem and is graded because the integral closure is a graded overring by~\cite[Lemma 1.6]{Cha-grintpruferlike}. Hence
$$R_{(\mathfrak{p})} \subseteq S \subseteq \bigcap_{Q \in  h\textnormal{-Spec}^1(S)}S_{(Q)} = \bigcap_{P \in  h\textnormal{-Spec}^1(R_{(\mathfrak{p})})} (R_{(\mathfrak{p})})_{(P)} = (R_{(\mathfrak{p})})^{\prime}.$$
So $S$ is integral over $R_{(\mathfrak{p})}$.

\bigskip

Let $P \in h\textnormal{-Spec}(R_{(\mathfrak{p})})$. Let
$$W = \{x \in R_{(\mathfrak{p})}\backslash P \, |\,x \textnormal{ is homogeneous}\}.$$
Then $(R_{(\mathfrak{p})})_{(P)} \subseteq S_W$ is an integral extension and satisfies going-down (because $R_{(\mathfrak{p})} \subseteq S$ does). Note that $PS_W \neq S_W$ by integrality. As previously noted, the integral closure of $R_{(\mathfrak{p})}$ is a Krull domain. So $S_W$ is gr-local by~\cite[Lemma 4.3]{EpSh-peri} and Lemma~\ref{gr-lem4.2-peri}. Then only one graded prime ideal of $S$ lies over $Q$. So $h\textnormal{-Spec}(S) \longrightarrow  h\textnormal{-Spec}(R_{(\mathfrak{p})})$ is injective.
\bigskip

Because $S$ is integral over $R_{(\mathfrak{p})}$, $h\textnormal{-Spec}(S) \longrightarrow h\textnormal{-Spec}(R_{(\mathfrak{p})})$ is surjective and hence is a bijection. By the Going-Up Theorem, the bijection is order-reflecting. Hence $S = R_{(\mathfrak{p})}$. Thus $R$ is a graded perinormal domain.
\end{proof}
\end{theorem}

\subsection{Graded domains that are graded perinormal domains}

In~\cite{MoZa-onpvmd}, Mott and Zafrullah introduced a class of integral domains that they called $P$-domains. In this section a notion of a graded $P$-domain is given and it is shown that every graded $P$-domain is a graded perinormal domain.

\begin{definition} A domain $R$ is a \emph{graded essential domain} (or simply a gr-essential domain) if 
$$R = \bigcap_{P \in I} R_{(P)}$$ 
where $I$ is a set of graded prime ideals of $R$ such that $R_{(P)}$ is a gr-valuation domain for all $P \in I$.
\end{definition}

\begin{definition} A domain $R$ is a \emph{graded $P$-domain} if it is a gr-essential domain such that every homogeneous localization is a gr-essential domain.
\end{definition}

Graded P\"{u}fer domains are graded $P$-domains by~\cite[Theorem 3.1 and Theorem 3.5]{ACZ-grpruf} and graded Krull domains are graded $P$-domains by~\cite[Theorem 5.12]{AnAn-divgr}.

\begin{proposition}\label{DR18-gr-2.1} Let $R$ be a gr-essential domain. If $S$ is a graded overring such that for each $P \in  h\textnormal{-Spec}(R)$ there exists $Q \in  h\textnormal{-Spec}(S)$ such that $Q \cap R = P$, then $S = R$.
\end{proposition}

\begin{proof} The proof is parallel to that of~\cite[Theorem 2.1]{DuRa-noteperi} using the appropriate graded substitutions and using~\cite[Theorem 2.3(4)]{AAC-grval} in place of~\cite[Theorem 26.1]{Gilmer}.
\end{proof}

\begin{proposition}\label{gr-thm2.2-noteperi} A graded $P$-domain is a graded perinormal domain.
\end{proposition}

\begin{proof} The proof is parallel to that of~\cite[Theorem 2.2]{DuRa-noteperi} making the appropriate graded substitutions, using the Proposition~\ref{gradedperiprop2.4} characterization of graded perinormality, Observation~\ref{gr-goingdown-surj} instead of~\cite[Lemma 2.2]{EpSh-peri}, Proposition~\ref{DR18-gr-2.1} instead of~\cite[Theorem 2.1]{DuRa-noteperi}, and Proposition~\ref{gradedRichThm2} instead of~\cite[Theorem 2]{Ric-genqr}. 
\end{proof}

\scalebox{0.8}{
$$\xymatrix{ \textnormal{graded Krull domain} \ar@{=>}[dr] & & \textnormal{graded P\"{u}fer domain} \ar@{=>}[dl]\\
& \textnormal{graded $P$-domain}\ar@{=>}[d] & \\
& \textnormal{graded perinormal domain} \ar@{=>}[dr]\ar@{=>}[dl]& \\
\textnormal{weakly normal} & & \textnormal{gr-$(R_1)$}}$$}

\subsection{Transfer of properties between $R$ and $R_0$}

We will use the following fact to prove the first result of this section.

\begin{proposition}[\cite{BrunsGubeladze}, Proposition 4.10(b)]\label{gr-quotientring->homogeneousms} Let $\Gamma$ be an arbitrary torsion free grading monoid $($i.e., torsion free commutative cancellative monoid$)$ and $R$ be a $\Gamma$-graded domain. Then every unit of $R$ is homogeneous.
\end{proposition}

\begin{proposition}\label{FQR-descent} Let $\Gamma$ be an arbitrary torsion free grading monoid $($i.e., torsion free commutative cancellative monoid$)$ such that $\Gamma \cap -\Gamma = \{0\}$ $($i.e., trivial group of units$)$. Let $R$ be a $\Gamma$-graded domain. Suppose that every flat graded overring of $R$ is a localization of $R$. Then every flat overring of $R_0$ is a localization of $R_0$.

\begin{proof} Suppose that every flat graded overring of $R$ is a localization of $R$. Let $A$ be a flat overring of $R_0$. So $AR$ is a flat graded overring of $R$. Thus $AR = R_W$ for some multiplicative subset $W$ of $R$. By Proposition~\ref{gr-quotientring->homogeneousms}, $W$ is a homogeneous multiplicative subset. Because $AR$ is $\Gamma$-graded and $\Gamma \cap -\Gamma = \{0\}$, any $w \in W$ must have $\deg(w) = 0$. So $W \subseteq R_0$. Note that clearly $(R_0)_W \subseteq S$ because $S = (AR)_0$. Let $s \in S$. Because $AR = R_W$, $s = \tfrac{r}{w}$ where $r \in R$ and $w \in W$. However, $W \subseteq R_0$ and $s = \tfrac{r}{w}$ is homogeneous and hence $\tfrac{r}{w} \in (R_0)_W$. Thus $A = (R_0)_W$.
\end{proof}
\end{proposition}

\begin{corollary} Let $\Gamma$ be an arbitrary torsion free commutative, cancellative monoid such that $\Gamma \cap -\Gamma = \{0\}$. Let $R$ be a $\Gamma$-graded domain. If $R$ is a graded globally perinormal domain and $R_0$ is perinormal, then $R_0$ is globally perinormal.

\begin{proof} Suppose that $R$ is a graded globally perinormal domain and $R_0$ is perinormal. Then every flat graded overring of $R$ is a localization of $R$ by Proposition~\ref{gradedperiprop6.2}. So by Proposition~\ref{FQR-descent}, every flat overring of $R_0$ is a localization of $R_0$. It follows from~\cite[Proposition 6.2]{EpSh-peri} that $R_0$ is globally perinormal since it is already perinormal by assumption.
\end{proof}
\end{corollary}

It would be of interest to determine whether $R$ being a graded perinormal domain implies that $R_0$ is a perinormal domain. An analogous question was posed in~\cite{EpSh-peri} of whether $A[X]$ being a perinormal domain implies that $A$ is perinormal. This has been considered by Dumitrescu and Rani and also by McCrady and Weston. The following are a couple of the results that have been obtained.

\begin{theorem}[\cite{DuRa-peripoly}, Theorem 1.1] Let $A$ be a Noetherian domain with integral closure $A^{\prime}$. Assume that the conductor $(A : A^\prime)$ has height at least two as an ideal and $A \slash (A: A^\prime)$ is zero-dimensional and not local. If $A[X]$ is perinormal, then $A$ is perinormal.
\end{theorem}

\begin{proposition}[\cite{McWe-periext}, Proposition 4.4] Let $A$ be an universally catenary domain with $\dim A \leq 2$. If $A[X]$ is perinormal, then $A$ is perinormal.
\end{proposition}

More results can be found in~\cite{DuRa-peripoly} and~\cite{McWe-periext} all of which include the assumption that $A$ is Noetherian.

Below we consider the graded analog of this question for polynomial rings. It follows directly from~\cite[Theorem 1]{Mca-gdpoly} that the only graded prime ideals of $A[X]$ lying over a prime $P$ of $A$ are either of the form $PA[X]$ or $PA[X] + XA[X]$.

\begin{proposition}\label{gr-peri-poly} Let $A$ be a $G$-graded domain and let $A[X]$ have the natural $G \times \mathbb{Z}_{\geq 0}$ grading. If $A[X]$ is a graded perinormal domain, then $A$ is a graded perinormal domain.

\begin{proof} Let $A$ be a $G$-graded domain and let $A[X]$ have the natural $G \times \mathbb{Z}_{\geq 0}$ grading. Note that it follows from~\cite[Theorem 1]{Mca-gdpoly} that the only graded prime ideals of $A[X]$ lying over a prime $P$ of $A$ are either of the form $PA[X]$ or $PA[X] + XA[X]$. For any $f \in A[X]$, $f$ can be expressed as $$f = f_{(g_{1},z_1)} + \cdots + f_{(g_n,z_n)}$$ where $f_{(g_i,z_i)} \in A_{g_i}X^{z_i}$. If $f \in A$, $f_{(g_i,z_i)} = 0$ for all $f_{(g_i,z_i)}$ where $z_i \neq 0$. It follows that if $Q \in h\textnormal{-Spec}(A[X])$, then $Q \cap A \in h\textnormal{-Spec}(A)$.
\bigskip 

Let $A[X]$ be a graded perinormal domain. Let $B$ be a gr-local gGD overring of $A$. It is clear that $B[X]$ is a graded overring of $A[X]$ with respect to the $G \times\mathbb{Z}_{\geq 0}$-grading.

\bigskip

\noindent\textbf{Claim 1:} $B[X]$ is a gr-local overring of $A[X]$.

\begin{proof} (Of Claim 1). Let $M\in  h\textnormal{-Max}(B[X])$. Then $M \cap B \in h\textnormal{-Spec}(B)$. Suppose that $M \cap B \not\in  h\textnormal{-Max}(B)$. Let $\mathfrak{m} \in  h\textnormal{-Max}(B)$ with $M \cap B \subset \mathfrak{m}$. Then $M^\prime \coloneqq \mathfrak{m}B[X] \supset M$ or $M^\prime \coloneqq \mathfrak{m}B[X] + XB[X] \supset M$ and $M^{\prime}\in h\textnormal{-Spec}(B[X])$ which contradicts the maximality of $M$. So $M \cap B \in  h\textnormal{-Max}(B)$. 

\bigskip
Now suppose $N \neq M \in  h\textnormal{-Max}(B[X])$. The only graded prime ideals of $B[X]$ lying over $M \cap B$ are $(M \cap B)B[X]$ and $(M \cap B)B[X] + XB[X]$. Similarly, the only graded prime ideals of $B[X]$ lying over $N \cap B$ are $(N \cap B)B[X]$ and $(N \cap B)B[X] + XB[X]$. Because $B$ is gr-local, $N \cap B = M \cap B \in  h\textnormal{-Max}(B)$. Because $M, N \in  h\textnormal{-Max}(B[X])$, it follows that $M = (M \cap B)B[X] + XB[X]$ and $N = (N \cap B)B[X] + XB[X]$. Hence $N = M$. So $B[X]$ is a gr-local overring of $A[X]$.
\end{proof}
\bigskip

The following observation with proof is included for the reader's convenience as it will be utilized to prove the claim that follows.

\bigskip
\noindent\textbf{Observation:} Let $A \subseteq B$ be an inclusion of rings and let $J$ be an ideal of $B$. Then 
$$(JB[X] + XB[X]) \cap A[X] = (J \cap A)A[X] + XA[X]$$
and 
$$(JB[X] \cap A[X]) = (J \cap A)A[X].$$

\begin{proof} (Of Observation). Note that $f= a_0 + a_1X + \cdots + a_nX^n \in (JB[X] + XB[X]) \cap A[X]$ if and only if $a_0 \in J$, $a_1, ..., a_n \in B$ and $a_0, ..., a_n \in A$, if and only if $a_0 \in J \cap A$ and $a_1, ..., a_n \in A$ if and only if $f \in (J\cap A)A[X] + X A[X]$. Hence $(JB[X] + XB[X]) \cap A[X] = (J \cap A)A[X] + XA[X]$.

\bigskip
Note that $f= a_0 + a_1X + \cdots + a_nX^n \in JB[X] \cap A[X]$ if and only if $a_0, ..., a_n \in J$ and $a_0,..., a_n \in A$ if and only if $a_0, ..., a_n \in J \cap A$ if and only if $f \in (J \cap A)A[X]$. Thus $JB[X] \cap A[X] =(J \cap A)A[X].$
\end{proof}

\noindent\textbf{Claim 2:} $B[X]$ satisfies gGD over $A[X]$.

\begin{proof} (Of Claim 2). Let $P, Q \in h\textnormal{-Spec}(A[X])$ with $P \subseteq Q$ and $\tilde{Q} \in h\textnormal{-Spec}(B[X])$ such that $\tilde{Q} \cap A[X] = Q$. Then $P \cap A, Q \cap A \in h\textnormal{-Spec}(A)$ with $P \cap A \subseteq Q \cap A$ and $\tilde{Q} \cap B \in h\textnormal{-Spec}(B)$ with $\tilde{Q} \cap A = \tilde{Q} \cap A[X] \cap A = Q \cap A$. Because $B$ satisfies gGD over $A$, there exists a prime ideal $\overline{P} \in h\textnormal{-Spec}(B)$ such that $\overline{P} \cap A = P \cap A$ and $\overline{P} \subseteq \tilde{Q} \cap B$. Because $\tilde{Q}$ is a graded prime ideal of $B[X]$, $\tilde{Q} = (\tilde{Q} \cap B)B[X]$ or $\tilde{Q} = (\tilde{Q} \cap B)B[X] + XB[X]$. 

\bigskip
\noindent\textbf{Case 1:} Suppose that $\tilde{Q} = (\tilde{Q} \cap B)B[X]$. Then $\overline{P}B[X] \subseteq \tilde{Q}$ because $\overline{P} \subseteq \tilde{Q} \cap B$. Note that in this case, 
$$Q = \tilde{Q} \cap A[X] = (\tilde{Q} \cap B)B[X] \cap A[X]= (\tilde{Q} \cap A)A[X] = (Q \cap A)A[X].$$
So $P = (P \cap A)A[X]$ because $P \in h\textnormal{-Spec}(A[X])$. Let $\tilde{P} \coloneqq \overline{P}B[X]$. Then $\tilde{P} \subseteq \tilde{Q}$ because $\overline{P} \subseteq \tilde{Q} \cap B$. Also 
$$\tilde{P} \cap A[X] = \overline{P}B[X] \cap A[X] = (\overline{P} \cap A)A[X] = (P \cap A)A[X] = P$$
since $\overline{P} \cap A = (P \cap A)$. So there exists $\tilde{P} \in h\textnormal{-Spec}(B[X])$ such that $\tilde{P} \subseteq \tilde{Q}$ and $\tilde{P} \cap A[X] = P$.

\bigskip
\noindent\textbf{Case 2:} Suppose that $\tilde{Q} = (\tilde{Q} \cap B)B[X] + XB[X]$. We have that $P = (P \cap A)A[X] + XA[X]$ or $P = (P \cap A)A[X]$ because $P$ is a graded prime ideal of $A[X]$ lying over $P \cap A$. If $P = (P \cap A)A[X] + XA[X]$, it follows that $Q = (Q \cap A)A[X] + XA[X]$. Let 
$$\tilde{P} \coloneqq \overline{P}B[X] + XB[X].$$
Then 
\begin{eqnarray*}
\tilde{P} \cap A[X] & =  &(\overline{P}B[X] + XB[X]) \cap A[X]\\ 
& = &(\overline{P} \cap A)A[X] + XA[X]\\
& = &(P \cap A)A[X] + XA[X]\\
& =& P.
\end{eqnarray*}
So $\tilde{P}  \subseteq \tilde{Q}$ and $\tilde{P} \cap A[X] = P$.

\bigskip
If $P = (P \cap A)A[X]$, let $\tilde{P} = \overline{P}B[X]$. Then clearly $\tilde{P} \subseteq \tilde{Q}$ and $\tilde{P} \cap A[X] = P$ as in Case 1.

\bigskip
It follows that $B[X]$ satisfies gGD over $A[X]$ which finishes the proof of the claim.
\end{proof}

Because $A[X]$ is a graded perinormal domain, it follows that $B[X] = (A[X])_W$ for some homogeneous multiplicative subset of $A[X]$. Because $B$ is an overring of $A$, it follows that $W \subseteq A$ and $B[X] = A_W[X]$. It follows that $B = A_W$ and hence $A$ is perinormal.
\end{proof}
\end{proposition}

\begin{corollary}\label{gr-peri-multivariable} Let $A$ be a $G$-graded domain. If $A[X_1, ..., X_n]$ is a graded perinormal domain with the natural $G \times \mathbb{Z}_{\geq 0}^n$-grading, then $A$ is a graded perinormal domain.

\begin{proof} This follows by applying Proposition~\ref{gr-peri-poly} a finite number of times.
\end{proof}
\end{corollary}

\begin{corollary}\label{gr-peri-multivariable-peri} Let $A$ be an integral domain. If $A[X_1, ..., X_n]$ is a graded perinormal domain with the natural $\mathbb{Z}_{\geq 0}^n$-grading, then $A$ is a perinormal domain.

\begin{proof} This follows by letting $G$ be the trivial group and using Corollary~\ref{gr-peri-multivariable}.
\end{proof}
\end{corollary}

Note that the proof of the above result utilizes that if $A \subseteq B$ satisfies going-down, then $A[X] \subseteq B[X]$ satisfies gGD. The example below shows that if $A \subseteq B$ satisfies going-down, $A[X] \subseteq B[X]$ need not satisfy going-down.

\begin{example} Let $B = \mathbb{C}\llbracket t \rrbracket$ and $A = \mathbb{R} + t\mathbb{C}\llbracket t \rrbracket$. So $A$ is the pullback of $B \twoheadrightarrow \mathbb{C} \hookleftarrow \mathbb{R}$. Because $\mathbb{C}$ is integral over $\mathbb{R}$, $B$ is integral over $A$ by~\cite[Theorem 1.1.1(7)]{Prufer}. By the example in~\cite{Mca-gdpoly} directly after~\cite[Theorem 3]{Mca-gdpoly}, $A[X] \subseteq B[X]$ is not unibranched and hence $A[X] \subseteq B[X]$ does not satisfy going-down by~\cite[Theorem 3]{Mca-gd}.
\end{example}

\bigskip
\begin{corollary}\label{gr-globally-peri-multivariable-globally-peri} Let $A$ be an integral domain. If $A[X_1, ..., X_n]$ is a graded globally perinormal domain with the natural $\mathbb{Z}_{\geq 0}^n$-grading, then $A$ is a globally perinormal domain.

\begin{proof} Suppose that $A[X_1, ..., X_n]$ is a graded globally perinormal domain. Then $A[X_1, ..., X_n]$ is a graded perinormal domain, so $A$ is perinormal by Proposition~\ref{gr-peri-multivariable-peri}. By Proposition~\ref{gradedperiprop6.2}, every graded flat overring of $A[X_1, ..., X_n]$ is a localization of $A[X_1, ..., X_n]$. Hence by Proposition~\ref{FQR-descent}, every flat overring of $A$ is a localization of $A$. Thus $A$ is globally perinormal by~\cite[Proposition 6.2]{EpSh-peri}.
\end{proof}
\end{corollary}

For perinormal domains, Epstein notes the following result.

\begin{observation}[Epstein, Personal Communication] If $I \subseteq S = k[X,Y]$ is an ideal and $S\slash I$ is a perinormal domain, then $S \slash I$ is regular and hence normal. In fact, if $I \neq (0)$, $S\slash I$ is a Dedekind domain.
\end{observation}

A similar result holds for graded perinormal domains.

\begin{proposition}\label{dim2-peri-then-normal} Let $R$ be a normal domain with $\dim(R) = 2$ and let $\mathfrak{p} \in \textnormal{Spec}(R)$. If $R \slash \mathfrak{p}$ is a graded perinormal domain, then $R \slash \mathfrak{p}$ is a normal domain.

\begin{proof} Let $R$ be a normal domain with $\dim(R) = 2$ and let $\mathfrak{p} \in \textnormal{Spec}(R)$. Suppose that $R \slash \mathfrak{p}$ is a graded perinormal domain. If $\mathfrak{p} = (0)$, then $R \slash \mathfrak{p} \cong R$ is a normal domain. So suppose $\mathfrak{p} \neq (0)$ and hence $\dim(R \slash \mathfrak{p}) \leq 1$. If $\dim(R\slash \mathfrak{p}) = 0$, then $R \slash \mathfrak{p}$ is a field and hence is a normal domain.

\bigskip
Suppose that $\dim(R\slash \mathfrak{p}) = 1$. Because $R \slash \mathfrak{p}$ is a graded perinormal domain, $R \slash \mathfrak{p}$ satisfies gr-$(R_1)$ by Proposition~\ref{gr-Prop3.2-peri}. If $\mathfrak{m} =(0) \in h\textnormal{-Max}(R)$, then every nonzero homogeneous element of $R$ is invertible and $R_{(\mathfrak{m})}$ is a graded valuation domain. So $R_{(\mathfrak{m})}$ is a graded valuation domain for every $\mathfrak{m} \in  h\textnormal{-Max}(R)$. Hence $R$ is a PVMD by~\cite[Theorem 3.1]{ACZ-grpruf} and hence is a normal domain.
\end{proof}
\end{proposition}

\section*{Acknowledgments}
I am very grateful to Neil Epstein for suggesting questions which led to this work and for helpful advice and input along the way. I also am grateful to Sean Lawton for a question which led to Corollary~\ref{gr-peri-multivariable} and Corollary~\ref{gr-globally-peri-multivariable-globally-peri}.

\bibliographystyle{amsalpha}
\bibliography{references}

\end{document}